\documentclass[11pt]{amsart}
    \usepackage{amsmath,amssymb,amsthm,amsfonts,amsrefs}
    \usepackage{geometry}
    \usepackage{graphicx}
    \usepackage[all]{xy}
    \usepackage[colorlinks, linkcolor=blue, anchorcolor=blue,
            citecolor=blue]{hyperref}

    \usepackage{enumerate}
    \usepackage{tikz}
    \usepackage{tikz-cd}
    \usepackage{xcolor,float}
    \usepackage{mathrsfs}
    \usepackage{caption}

    \usepackage{fancyhdr}
    \usepackage{mathrsfs}
    \newtheorem{definition}{Definition}[section]
    \newtheorem{theorem}{Theorem}[section]
    \newtheorem{proposition}[theorem]{Proposition}
    \newtheorem{lemma}[theorem]{Lemma}
    \newtheorem{corollary}[theorem]{Corollary}
    
    \newtheorem{example}{Example}[section]
    \newtheorem{remark}[example]{Remark}

\newcommand{\bD}{\mathbb{D}}

\newcommand{\bR}{\mathbb{R}}
\newcommand{\bZ}{\mathbb{Z}}

\begin{document}

\title{Existence of Generating Families on Lagrangian Cobordisms}
\date{}
\author{Wenyuan Li}
\address{Department of Mathematics, Northwestern University.}
\email{wenyuanli2023@u.northwestern.edu}
\maketitle

\begin{abstract}
    For an embedded exact Lagrangian cobordism between Legendrian submanifolds in the 1-jet bundle, we prove that a generating family linear at infinity on the Legendrian at the negative end extends to a generating family linear at infinity on the Lagrangian cobordism after stabilization if and only if the formal obstructions vanish. In particular, a Lagrangian filling with trivial stable Lagrangian Gauss map admits a generating family linear at infinity.
\end{abstract}

\section{Introduction}

    Generating families, as a generalization of defining functions of graphical Lagrangians and Legendrians in cotangent bundles and 1-jet bundles \cite{Hormander,WeinLecture,GuiStern}, have been an important tool in the studies of exact Lagrangian submanifolds in cotangent bundles and respectively Legendrian submanifolds in 1-jet bundles \cite{Sikorav,ViterboGen,GroEliashGF,TrayGen,JordanTraynor}. In particular, generating families have been used in the studies of Lagrangian submanifolds in the symplectization of $J^1M$ \cite{ChekanovGF,Chaperon,GroEliashGF}. Sabloff-Traynor \cite{Genfamily} and Bourgeois-Sabloff-Traynor \cite{BST-GF} have used generating family techniques in the study of Lagrangian cobordisms between Legendrian submanifolds.

    However, it is in general very difficult to construct generating families for an arbitrary Lagrangian or Legendrian. To the best of our knowledge, the Lagrangians/Legendrians we know that generating families exist globally for systematic reasons are the ones that are Hamiltonian isotopic to the zero section \cite{Sikorav,ViterboGen,TheretViterbo,ChekanovGF,GroEliashGF} or (sub)graphical Lagrangians \cite{GroEliashGF}, nearby Lagrangians in $T^*M$ with trivial stable Lagrangian Gauss maps \cite{TwistGF} and arbitrary nearby Lagrangians in $T^*S^n$ \cite{KraghR2n,TwistGF}. There have also been constructions for certain immersed Lagrangian fillings of Legendrians \cite{BST-GF,Pezzimenti}, Lagrangian traces of Legendrian isotopies and certain Lagrangian handle attachments on Legendrians \cite{BST-GF}.

    In this note, we consider generating families on Lagrangian cobordisms between Legendrian submanifolds. This is a non-symmetric relation between Legendrian submanifolds \cite{Chantraine,ChantraineNonsym}, and results in relative symplectic field theory \cite{SFT} imply that Lagrangian cobordisms induce maps of the Legendrian contact homologies from the positive end to the negative end \cite{Ekcobordism,EHK}, and functors between augmentation categories from the negative end to the positive end. Since in low dimensions generating families induce augmentations by using rulings \cite{ChekPushkar,FuchsRuling,FuchsIshkhanovRuling,SabRuling,FuchsRutherGF} and Morse complex sequences \cite{HR-Ruling,RutherSullivanGF}, one would expect that a generating family on the Legendrian at the negative end extends to the positive end.

    Our main theorem states that generating families on the Legendrian at the negative end indeed extends to the Lagrangian cobordism and hence the positive end, once the necessary formal obstruction vanishes.

\begin{theorem}
    Let $L \hookrightarrow J^1M \times \bR_{>0} \xrightarrow{\sim} T^*(M \times \bR_{>0})$ be an exact Lagrangian cobordism between Legendrians in $J^1M$ from $\Lambda_-$ to $\Lambda_+$. Let $f_-$ be a generating family for $\Lambda_-$ linear at infinity. Then $f_-$ extends to a generating family for $L$ linear at infinity up to stabilizations if and only if the stable Lagrangian Gauss map $G: L \rightarrow U/O$ is null homotopic and the null homotopy on $\Lambda_-$ determined by $f_-$ extends to a null homotopy on $L$. In particular, when the above conditions hold, $\Lambda_+$ admits a generating family linear at infinity.
\end{theorem}

\begin{remark}
    We expect that for any $\bD^k$-family of generating families linear at infinity for $\Lambda_-$, when the $\bD^k$-family of null homotopies of the Lagrangian Gauss map extends to $L$, the generating families will extend to a $\bD^k$-family of generating families linear at infinity for $L$. However, we will not work out the details and give a complete proof in this paper.
\end{remark}

    Generating families have been used in many previous literature to study Lagrangian cobordisms between Legendrians \cite{Genfamily,SabSullivanGF,BST-GF,Pezzimenti,Limouzineau}. However, to our knowledge, we need to assume the existence of a compatible generating family on the Lagrangian cobordism. With the theorem above, one can replace those existence assumptions by the purely formal conditions on the Lagrangian Gauss map and the classifying map of generating families.

    The first corollary is that for exact Lagrangian concordances, a generating family linear at infinity on Legendrian at the negative end always extends to the Lagrangian concordance.

\begin{corollary}
    Let $L \hookrightarrow J^1M \times \bR_{>0} \xrightarrow{\sim} T^*(M \times \bR_{>0})$ be an exact Lagrangian concordance between closed Legendrians from $\Lambda_-$ to $\Lambda_+ \hookrightarrow J^1M$. Then $L$ admits a generating family linear at infinity if and only if $\Lambda_-$ admits a generating family linear at infinity.
\end{corollary}

    The second corollary that may be of particular interest is the case of Lagrangian fillings. Our result implies that any exact Lagrangian filling with trivial stable Lagrangian Gauss map admits a generating family linear at infinity.

\begin{corollary}
    Let $L \hookrightarrow J^1M \times \bR_{>0} \xrightarrow{\sim} T^*(M \times \bR_{>0})$ be an exact Lagrangian filling of $\Lambda \hookrightarrow J^1M$. Then $L$ admits a generating family linear at infinity if and only if the stable Lagrangian Gauss map $G: L \rightarrow U/O$ is trivial. In particular, when $\Lambda$ has a Lagrangian filling $L$ with trivial Lagrangian Gauss map, $\Lambda$ admits a generating family linear at infinity.
\end{corollary}

    We explain the concepts involved in the main results, starting from the definition of generating families for Lagrangian and Legendrian submanifolds.

\begin{definition}\label{def:linear}
    Let $\Lambda \hookrightarrow T^*M$ be an exact Lagrangian embedding. Then $f: M \times \bR^n \rightarrow \bR$ is called a generating family for $\Lambda$ if the graph of the differential $\Lambda_{df} \pitchfork T^*M \times \bR^n \times \{0\}$ and
    \begin{align*}
    \Lambda = \pi_{T^*M} (\Lambda_{df} \cap T^*M \times \bR^n \times \{0\}) = \{(x, \xi) \mid \exists\,u \in \bR^n, \partial_x f(x, u) = \xi, \partial_u f(x, u) = 0\}.
    \end{align*}
    Let $\Lambda \hookrightarrow J^1M$ be a Legendrian embedding. Then $f: M \times \bR^n \rightarrow \bR$ is called a generating family for $\Lambda$ if the graph of the 1-jet $\Lambda_{j^1f} \pitchfork J^1M \times \bR^{n} \times \{0\}$ and
    \begin{equation*}
    \resizebox{\textwidth}{!}{$\Lambda = \pi_{J^1M} (\Lambda_{j^1f} \cap J^1M \times \bR^{n} \times \{0\}) = \{(x, \xi, z) \mid \exists\,u \in \bR^n, \partial_x f(x, u) = \xi, \partial_u f(x, u) = 0, f(x, u) = z\}.$}
    \end{equation*}
    $f$ is a generating function linear at infinity if $f(x, -): \bR^n \rightarrow \bR$ is a linear function for every $x \in M$ outside a proper subset with respect to the projection onto $M$.
\end{definition}

    Given a generating family, we can stabilize it by introducing extra dimensions and adding quadratic forms on the extra factors. Let $f: M \times \bR^n \rightarrow \bR$ be a generating family. Then the stabilization of $f$ is the generating function
    $$\overline{f}: M \times \bR^n \times \bR^m \rightarrow \bR, \;\; \overline{f}(x, u_0, u_1) = f(x, u_0) + Q(u_1)$$
    for a quadratic form $Q: \bR^m \rightarrow \bR$. We will prove that after deforming the functions outside compact subsets, stabilizations preserve linear at infinity property.

    Then we explain the formal obstructions that appear in the statement of our theorems. The first obstruction is to trivialize the stable Lagrangian Gauss map of $\Lambda$. Consider the Lagrangian fibration of $T^*M$ by cotangent fibers. We will identify the Lagrangian Grassmannian over $T^*M$ with the trivial bundle $U/O \times T^*M \to T^*M$.

\begin{definition}
    Let $\Lambda \rightarrow T^*M$ be an exact Lagrangian immersion. Then the stable Lagrangian Gauss map is defined by the section
    $$G: \Lambda \rightarrow U/O, \;\; p \mapsto [T_p\Lambda].$$
\end{definition}

    It is shown by Giroux \cite{GirouxGF} and Latour \cite{LatourGF} that the necessary and sufficient condition of the existence of a germ of generating family on an exact Lagrangian immersion $\Lambda$ is that the stable Lagrangian Gauss map is trivial, and each generating family $f$ determines a null homotopy of the Lagrangian Gauss map $G: \Lambda \to U/O$. Then the second obstruction is to extend the null homotopy on $\Lambda$ to a null homotopy on $L$.

    When concatenating two null homotopies of the stable Lagrangian Gauss map, we will obtain a based loop $\Lambda \to \Omega(U/O)$, where $\Omega(U/O) \simeq \bZ \times BO$ by Bott periodicity. Thus we can characterize the null homotopy through the following construction, where we identify a subspace $V \subset \bR^n$ of codimension $r$ as a point $(r, [V]) \in \bZ \times BO$.

\begin{definition}\label{def:classify-intro}
    Let $\Lambda \rightarrow T^*M$ be an exact Lagrangian immersion generated by a given reference generating family $g: M \times \bR^n \rightarrow \bR$. Then the classifying map of any generating family $f$ is defined as a map between pointed spaces by
    \begin{align*}
    \Delta(f): \Lambda \rightarrow \bZ \times BO, \; (x, \partial_x f(x, u)) \mapsto [V_{<0}f(x, u) \oplus V_{\geq 0}g(x, u)],
    \end{align*}
    where $V_{<0}f$ and $V_{\geq 0}f$ are the negative and non-negative eigenspace of the Hessian $\partial^2_u f$.
\end{definition}
\begin{remark}
    Latour showed that $V_{<0}f(x, u) \oplus V_{\geq 0}g(x, u)$ defines a vector bundle over $\Lambda$, which is not obvious from the construction. For any $(x, u) \in M \times \bR^n$ that projects onto $\Lambda$, the value of $\Delta(f_-)$ at the point can also be written as $[V_{<0}f(x, u)] - [V_{< 0}g(x, u)]$.
\end{remark}
\begin{remark}
    Since $\Delta$ is a map between pointed spaces,  we require that the base point is sent to $0\in \bZ \times BO$. When $\Lambda$ is connected, the image of $\Delta(f)$ is contained in $0 \times \mathrm{Gr}(\bR^\infty, \bR^\infty)$.
\end{remark}

    The necessary condition for one to extend a germ of generating family is that the corresponding classifying map extends. It is proved by Latour \cite{LatourGF} that the germs of generating families are indeed classified by $[\Lambda, \Omega(U/O)] \simeq [\Lambda, \bZ \times BO]$ up to stabilizations. Hence this is also the sufficient condition to extend germs of generating families.

\begin{remark}
    We note that the first few formal obstructions are well known invariants for Lagrangians and Legendrians. The first obstruction for the Lagrangian Gauss map to vanish is the Maslov class $\mu(\Lambda) \in H^1(\Lambda; \bZ)$ and the first order part of the classifying map from $L$ to $\bZ$ is the choice of a Maslov potential on all connected components. We will explain in Appendix \ref{sec:knot} that for Legendrian links, these are the only obstructions.
\end{remark}

    Knowing the result of Giroux \cite{GirouxGF} and Latour \cite{LatourGF}, one can easily notice that the main difficulty is to extend the germ of generating family to a generating family linear at infinity on $M \times \bR^N$. We follow the idea in the recent work of Abouzaid-Courte-Guillermou-Kragh \cite{TwistGF} inspired by works in microlocal sheaves of Guillermou \cite{Gui}. Given a germ of generating families on $\Lambda$, one constructs a generating family linear at infinity on the double copy given by Reeb push-off $T_{-h}(\Lambda) \cup T_h(\Lambda)$, and then separates the two copies $T_{-h}(\Lambda)$ and $T_h(\Lambda)$.

    For exact Lagrangian cobordisms $L$ between Legendrians, we do not always obtain a generating family linear at infinity on the double copy $T_{-h}(L) \cup T_h(L)$ due to lack of a Weinstein tubular neighbourhood of $L$ with positive radius. However, once there is a generating family linear at infinity at the negative end $\Lambda_-$, we will be able to carry out the rest of the doubling construction. This is inspired by our previous work in microlocal sheaves for Lagrangian cobordisms \cite{LiCobordism2} (though we need to emphasize that mathematically neither result implies the other).

\begin{remark}
    Like the construction of Abouzaid-Courte-Guillermou-Kragh \cite{TwistGF}, we believe that our construction also works for twisted generating families. However, in the case of nearby Lagrangians considered in their paper, twisted generating families produce interesting restrictions of the Lagrangian. It seems yet unclear what one can obtain by using twisted generating families on Lagrangian cobordisms.
\end{remark}

    Let us remark that our result should potentially generalize to generating families that are not linear at infinity. However, in those cases there might be extra difficulties when cutting off the generating families in Section \ref{sec:cutoff}.

\subsection*{Acknowledgement}
    We would like to thank Emmy Murphy and Eric Zaslow for helpful discussions on the project. We would also like to thank Daniel Alvarez-Gavela for helpful discussions and St\'ephane Guillermou, Thomas Kragh, Joshua Sabloff and Lisa Traynor for their interest in this project. We thank Joshua Sabloff for reading a earlier version of the draft and providing helpful feedbacks, thank Sylvain Courte for helpful discussions and comments on the draft, and thank the anonymous referee for carefully reading the paper and providing detailed feedback throughout. After the manuscript is posted, we learned that Sylvain and St\'ephane are aware of this result as well.

\section{Proof of the Theorem}

\subsection{Germs of generating families}
    The problem of constructing germs of generating families turns out to be purely homotopy theoretic. In this section, we explain why the formal condition ensures the existence of a germ of generating family.

\begin{definition}
    Let $\iota: L \rightarrow J^1M$ be a Legendrian immersion. Then a germ of generating function is a pair $(U, f)$ where $U \subseteq M \times \bR^n$ is an open subset $f: U \rightarrow \bR$ is a smooth function such that
    $$\Lambda_{j^1f} \pitchfork J^1M \times \bR^n \times \{0\}.$$
    Then $(U, f)$ is a germ of generating family for $L$ if for $\Sigma_f = \Lambda_{j^1f} \cap J^1M \times \bR^n \times \{0\}$, there is a diffeomorphism $\varphi: L \xrightarrow{\sim} \Sigma_f$ such that $\iota = \pi_{J^1M} \circ \varphi$, and in particular
    \begin{align*}
    \iota(L) \cong \pi_{J^1M} (\Sigma_f) = \{(x, \partial_x f(x, u), f(x, u)) \mid \exists\, u \in \bR^k, \partial_u f(x, u) = 0, (x, u) \in U\}.
    \end{align*}
\end{definition}
\begin{remark}
    When $L \hookrightarrow J^1M$ is a Legendrian embedding, we will omit the obvious choice of the diffeomorphism $\varphi: L \xrightarrow{\sim} \Sigma_f$.
\end{remark}

    Denote the Legendrian embedding by $\varphi: L \hookrightarrow J^1M$. Consider a smooth map $\psi: L \rightarrow \bR^n$ such that $\pi \times \psi: L \hookrightarrow M \times \bR^n$ is an embedding and $U$ is an open neighbourhood of $(\pi \times \psi)(L) \subseteq M \times \bR^n$. A germ of generating family for $L$ is equivalent to a germ of graphical Legendrian embedding
    $$\theta_U: U \hookrightarrow J^1(M \times \bR^k)$$
    that transversely intersects the coisotropic $J^1(M) \times \bR^k$ along the isotropic embedding 
    $$\varphi \times d\psi: L \hookrightarrow J^1(M \times \bR^n).$$ 
    such that $\theta_U|_L = \varphi \times d\psi$. One can further show that this is equivalent to a Lagrangian distribution that contains the tangent bundle of the isotropic embedding, which then becomes a homotopy theoretic problem.

\begin{theorem}[Giroux \cite{GirouxGF}]\label{thm:Giroux}
    Let $L \rightarrow J^1M$ be a Legendrian immersion. Then $L$ has a germ of generating family if and only if the stable Lagrangian Gauss map $L \rightarrow U/O$ is homotopically trivial.
\end{theorem}

    Then we define equivalences of germs of generating families following \cite[Definition~IV.1.4 \& 6]{LatourGF} in order to state the classification theorem of Latour.

\begin{definition}
    Let $L \rightarrow J^1M$ be a Legendrian immersion, and $(U, f)$, $(U', f')$ be germs of generating families for the Legendrian $L$. Then $(U, f)$ and $(U', f')$ are called strictly equivalent if there are open subsets $V \subseteq U$, $V' \subseteq U'$, and a diffeomorphism $H: V \xrightarrow{\sim} V'$ such that $H(\Sigma_{f}) = \Sigma_{f'}$, $\pi_M \circ H = \pi_M|_{V}$, and
    $$f'|_{V'} = f \circ H.$$
    $(U, f)$ and $(U', f)$ are called stably equivalent if for some stabilizations
    \begin{align*}
    \overline{f}: U \times \bR^m \rightarrow \bR, &\;\, \overline{f}(x, u_1, u_2) = f(x, u_1) + Q(u_2),\\
    \overline{f}{}': U' \times \bR^{m'} \rightarrow \bR, &\;\, \overline{f}{}'(x, u_1, u_2) = f'(x, u_1) + Q'(u_2)
    \end{align*}
    with nondegenerate quadratic forms $Q, Q': \bR^m \rightarrow \bR$, there is a strict equivalence between $(U \times \bR^m, \overline{f})$ and $(U' \times \bR^{m'}, \overline{f}{}')$.
\end{definition}

    We now explain the correspondence between the germ of generating families and the homotopy classes $[L, \Omega(U/O)]$. Using Bott periodicity, we know that
    $$\Omega(U/O) \simeq \bZ \times BO \simeq \bZ \times \mathrm{Gr}(\bR^\infty, \bR^\infty).$$
    For any subspace $V \subset \bR^n$ of codimension $r$, we will regard it as a point $(r, [V]) \in \bZ \times BO$. In the proof of Giroux's theorem, we deformed the Lagrangian Gauss map such that it is tangent to the isotropic embedding $L \hookrightarrow J^1(M \times \bR^{n})$. However, the behaviour along the $\bR^n$-direction of the embedding $U \hookrightarrow J^1(M \times \bR^{n})$ may differ, and we can obtain a classification result of the germs of generating families by the behaviour along $\bR^n$-direction.

    For a generating family $f$ for $L \hookrightarrow J^1M$, the negative eigenspace of $f$ on the fiberwise critical locus $\Sigma_f$ as a subspace in $\bR^n$. However, in general, the dimension of the negative eigenspace is not a constant and does not define a vector bundle or a continuous map $L \to BO$. Latour resolved the issue by considering the sum of the negative eigenspace of a generating family $f$ and the positive eigenspace of a reference generating family $g$, which turns out to be always a direct sum \cite[Proposition III 4.6 \& IV 2.4]{LatourGF}.

\begin{definition}
    Let $f$ and $g$ be germs of generating families for $L$ such that $\Sigma_f = \Sigma_g$. Let $V_{<0}f(x, u)$ and $V_{\geq 0}f(x, u)$ be the negative and non-negative eigenspace of the Hessian $\partial^2_u f(x, u)$ at $(x, u) \in \Sigma_f$. Denote by $(x_0, u_0) \in \Sigma_f$ the lift of the base point in $L$. We define the classifying map of $f$ with respect to the reference class $g$ as
    $$\Delta(f) = \Delta(f, g) = [V_{<0}f \oplus V_{\geq 0}g] - [V_{<0}f(x_0, u_0) \oplus V_{\geq 0}g(x_0, u_0)] \in [L, \bZ \times BO].$$
\end{definition}

\begin{theorem}[Latour \cite{LatourGF}*{Theorem IV.1.10}]\label{thm:Latour}
    Let $L \rightarrow J^1M$ be a Legendrian immersion such that the stable Lagrangian Gauss map $L \rightarrow U/O$ is homotopically trivial. Then the stable equivalence classes of generating families on $L$ are in bijection with homotopy classes of pointed spaces $[L, \Omega(U/O)]$.
\end{theorem}
\begin{remark}
    Latour \cite{LatourGF}*{Theorem IV.1.10} assumed that $L$ is connected and proved that germs of generating families are stably classified by the stable homotopy classes $[L, BO]$. However, when $L$ is disconnected, since stabilizations by quadratic forms change the indices of critical points on all connected components simultaneously, the classification is given by homotopy classes of pointed spaces $[L, \bZ \times BO]$ where the component containing the base point $x_0 \in L$ is sent to $0 \times BO$. Alternatively, we could only allow stabilizations by quadratic forms of signature $0$ to fix the component of the image in $\bZ$.
\end{remark}

    In order to compare arbitrary germs of generating families, we need the following lemma.

\begin{lemma}[Latour \cite{LatourGF}*{Lemma IV.4.1}]\label{lem:common-domain}
    Let $L \rightarrow J^1M$ be a Legendrian immersion, and $(U, f)$, $(U', f')$ be germs of generating families for $L$. Then there exist stabilizations $(\overline{U}, \overline{f})$, $(\overline{U}{}', \overline{f}{}')$ and a diffeomorphism $H: \overline{U} \xrightarrow{\sim} \overline{U}{}'$ such that $(\overline{U}, \overline{f})$, $(\overline{U}, \overline{f}{}' \circ H)$ have common critical locus $$\Sigma_{\overline{f}} = \Sigma_{\overline{f}{}' \circ H}.$$
\end{lemma}

    Given any two germs of generating families $f$ and $f'$, by the above lemma, one can assume that after stabilizations they are defined on the same domain. When $\Delta(\overline{f}) - \Delta(\overline{f}{}') = 0$, there is a homotopy between the classifying maps $L \rightarrow \bZ \times BO$ of $\overline{f}$ and $\overline{f}{}'$. Therefore, by Bott periodicity we can show that the null homotopies of the Lagrangian Gauss map $L \times [0, 1] \rightarrow U/O$ determined by $\overline{f}$ and $\overline{f}{}'$ are homotopic. Using the existence theorem of generating families for $L \times [0, 1]$, we get the following lemma.

\begin{lemma}[Latour \cite{LatourGF}*{Proof IV.4.5}]\label{lem:germ-htpy}
    Let $L \rightarrow J^1M$ be a Legendrian immersion, and $(U, f)$, $(U, f')$ be germs of generating families for $L$ with common critical locus $\Sigma_f = \Sigma_{f'}$. Suppose $\Delta(f) - \Delta(f') = 0$. Then after stabilization there is a family of germs of generating functions $(\overline{U}, \overline{f}_t)$ for $L$ such that $\overline{f}_0 = \overline{f}$ and $\overline{f}_1 = \overline{f}{}'$.
\end{lemma}

    Then using Moser argument and isotopy extension theorem \cite{LatourGF}*{Lemma IV.4.4}, from the homotopy of generating families $(\overline{U}, \overline{f}_t)$, one can easily define an isotopy of manifolds $H_t$ such that $\overline{f}_t = \overline{f} \circ H_t$. This is the key idea in the proof of Latour's theorem.

    Using Latour's theorem, we can prove the following extension theorem for germs of generating families.

    Suppose $L$ is a Legendrian and the stable Lagrangian Gauss map $L \rightarrow U/O$ is trivial. We fix a reference class of germ of generating family $g$ for $L$, and then for any generating family $f$ for $L$, define the classifying map
    $$\Delta(f) = [V_{<0}f \oplus V_{\geq 0}g] - [V_{<0}f(x_0, u_0) \oplus V_{\geq 0}g(x_0, u_0)] \in [L, \bZ \times BO].$$
    For an open submanifold $L_0 \subseteq L$, we restrict the reference class $g$ to $L_0$, and given a generating family $f_0$ for $L_0$, define the classifying map
    $$\Delta(f_0) = [V_{<0}f_0 \oplus V_{\geq 0}g|_{L_0}] - [V_{<0}f_0(x_0, u_0) \oplus V_{\geq 0}g(x_0, u_0)|_{L_0}] \in [L_0, \bZ \times BO].$$

\begin{proposition}\label{prop:germ-extend}
    Let $L \rightarrow J^1(M)$ be a compact Legendrian immersion. Then a generating family $f_0$ for an compact submanifold $L_0 \subseteq L$ with smooth boundary extends to $L$ after stabilization if and only if the stable Lagrangian Gauss map $L \rightarrow U/O$ is homotopically trivial and the homotopy class associated with the generating family $\Delta({f_0}): L_0 \rightarrow \bZ \times BO$ extends to $L$.
\end{proposition}
\begin{proof}
    When the Lagrangian Gauss map $L \rightarrow U/O$ is stably trivial, by Giroux's Theorem \ref{thm:Giroux} we know that $L$ admits a generating family, and by Latour's Theorem \ref{thm:Latour} we know that (for a given choice of reference class of generating families) the stable equivalence classes of generating families are in bijection with $[L, \bZ \times BO]$.

    Consider the generating family $f_0$ on a compact submanifold $L_0 \subseteq L$, and $U_0$ an open neighbourhood of $L_0$ such that $L_0 = U_0 \cap L$. Let $L_0'$ be an open tubular neighbourhood of $L_0 \subseteq L$ which deformation retracts onto $L_0$, and $U_{0}'$ be an open neighbourhood $U_{0} \subseteq U$ which deformation retracts to $U_0$ such that $L_0' = U_0' \cap L$. Then $f_0$ extends to a germ of generating family $f_0'$ on $L_0'$. Under the deformation retraction $L_0' \xrightarrow{\sim} L_0$, the associated homotopy class $\Delta({f_0}): L_0 \rightarrow \bZ \times BO$ is identical as $\Delta({f_0}'): L_0' \rightarrow \bZ \times BO$. Assume that $\Delta({f_0}')$ extends to $\Delta(f_1): L \rightarrow \bZ \times BO$ which by Latour's Theorem \ref{thm:Latour} corresponds to the generating family $f_1$ defined on an open neighbourhood $U$ of $L$. Restricting to $U_0'$, we can apply Latour's Theorem \ref{thm:Latour}, Lemma \ref{lem:common-domain} and Lemma \ref{lem:germ-htpy} so that after stabilization, $\overline{f}_1|_{\overline{U}{}'_0}$ is homotopic to $\overline{f}{}'_0$ through a family $\overline{f}_t$, $0 \leq t \leq 1$. Let $\rho: \overline{U} \rightarrow \bR$ be a cut-off function such that $\rho|_{\overline{U}_0} = 0$ and $\rho|_{\overline{U} \backslash \overline{U}_0'} = 1$. Then
    $$\overline{f}(x, u) = \overline{f}_{\rho(x, u)}'(x, u) : \overline{U} \rightarrow \bR$$
    is a generating family on $L$ stably equivalent to $f_0$ on the open neighbourhoods $U_0$ of $L_0$.
\end{proof}

\subsection{Generating families linear at infinity}
    Any Morse-theoretic argument of generating families will not work without control on the behaviour of the generating families away from the critical locus. This is why we need the linear at infinity condition. Following \cite{TwistGF}, we introduce generating families that are weakly linear at infinity. Unlike Definition \ref{def:linear}, for generating families that are weakly linear at infinity, we only require linearity outside compact subsets on some $\bR^n$-factor of $\bR^{m+n}$. Alternatively, one can also simply deal with generating functions linear and quadratic at infinity following Jordan-Traynor \cite{JordanTraynor}.

\begin{definition}[\cite{TwistGF}*{Definition 3.1}]
    Let $f: M \times \bR^{m+n} \rightarrow \bR$ be a generating function. Then $f$ is called a function weakly linear at infinity if outside a compact subset in $\bR^{m+n}$,
    $$f(x, u_1, u_2) = g(x, u_1) + L(x, u_2)$$
    where $L(x, -): \bR^n \rightarrow \bR$ is a linear function. Equivalently, $f$ is called weakly linear at infinity if there is a function $\epsilon: M \times \bR^{m+n} \rightarrow \bR$ where $\pi_M: \mathrm{supp}(\epsilon) \rightarrow M$ is proper such that
    $$f(x, u_1, u_2) = \epsilon(x, u_1, u_2) + g(x, u_1) + L(x, u_2).$$
\end{definition}

\begin{lemma}
    Let $f: M \times \bR^{m+n} \rightarrow \bR$ be a generating function weakly linear at infinity. Then there is a homotopy of functions $f_t$ such that $f_0 = f$ and $f_1$ is a function strongly linear at infinity such that the fiberwise critical locus $\Sigma_{f_t} = \Sigma_f$ and $f_t = f$ on a neighbourhood of the critical locus.
\end{lemma}
\begin{proof}
    Let $\rho: M \times \bR^{m+n} \rightarrow \bR$ be a smooth cut-off function such that $\rho|_{\mathrm{supp}(\epsilon)} \equiv 1$ and $\pi_M: \mathrm{supp}(\rho) \rightarrow M$ is proper. Define $f_t(x, u_1, u_2) = \epsilon(x, u_1, u_2) + (1 - t + t\rho(x, u_1, u_2))g(x, u_1) + L(x, u_2)$. This is the required homotopy.
\end{proof}

    Comparing Definition \ref{def:linear} of functions strictly linear at infinity, the advantage of functions weakly linear at infinity is that they are invariant under stabilizations.

\begin{definition}[\cite{TwistGF}*{Definition 3.3}]
    Let $f: M \times \bR^{m+n} \rightarrow \bR$ be a generating function weakly linear at infinity of the form
    $$f(x, u_1, u_2) = \epsilon(x, u_1, u_2) + g(x, u_1) + L(x, u_2)$$
    where $\pi_M: \mathrm{supp}(\epsilon) \rightarrow M$ is proper and $\sup|\epsilon| \leq b$.

    Let $\mathcal{Q}_k$ be the space of quadratic forms on $\bR^k$, $Q: M \rightarrow \mathcal{Q}_k$ be a family of quadratic forms, and $\chi: \mathcal{Q}_k \times \bR^k \rightarrow \bR$ be a cut-off function such that $\chi_q(v) = \chi(q, v)$ is compactly supported, $\chi_q \equiv 1$ near $0 \in \bR^m$ and $|D_v \chi_q| < |D_v q|$. Then the stabilization of $f$ by $Q$ with compact cut-off is defined by
    \begin{align*}
    &\overline{f}_\textit{cpt}:  M \times \bR^{m+n} \times \bR^k \rightarrow \bR, \\
    \overline{f}_\textit{cpt}(x, u_1, u_2, v) = & \, \chi_{Q(x)}(b^{-1}v)\epsilon(x, u_1, u_2) + g(x, u_1) + Q(x, v) + L(x, u_2).
    \end{align*}
\end{definition}
\begin{remark}
    We refer to \cite{TwistGF}*{Lemma 3.2} for the details on the existence of the choice of cut-off function $\chi: \mathcal{Q}_k \times \bR^k \rightarrow \bR$ that satisfy compatibility conditions.
\end{remark}

\begin{lemma}[Sabloff-Traynor \cite{Genfamily}*{Lemma 3.8}, Abouzaid-Courte-Guillermou-Kragh \cite{TwistGF}*{Lemma 3.4}]
    Let $f: M \times \bR^{m+n} \rightarrow \bR$ be a generating function weakly linear at infinity and $Q: M \rightarrow \mathcal{Q}_k$ be a family of quadratic forms. Let $\overline{f}_\textit{cpt}$ be the stabilization by $Q$ with compact cut-off and $\overline{f}$ be the stabilization by $Q$. Then there is a homotopy $\overline{f}_t$ with $\overline{f}_0 = \overline{f}_\textit{cpt}$ and $\overline{f}_1 = \overline{f}$ such that $\overline{f}_t$ have the same critical locus and agree on a neighbourhood of the critical locus.
\end{lemma}

    Given the above lemma, we will no longer distinguish generating families that are weakly linear and strongly linear at infinity, nor will be distinguish stabilization of $f$ by $Q$ with compact cut-off and standard stabilizations of $f$ by $Q$.

\subsection{Generating families on the doubling}
    Let $L$ be the Legendrian lift of an exact Lagrangian cobordism from $\Lambda_-$ to $\Lambda_+$ \cite{Chantraine,EHK}. Given a generating family linear at infinity on the negative end $\Lambda_-$ that extends to a germ of generating family on $L$, we use the doubling construction to get a generating family linear at infinity for the double copies $T_{-\epsilon}(L) \cup T_\epsilon(L)$ given by the Reeb pushoff after modifying the construction of Abouzaid-Courte-Guillermou-Kragh \cite{TwistGF}. This is the key construction in our paper.

    First, recall that there is a strict contactomorphism between the contactization of the symplectization of $J^M$ and $J^1(M \times \bR_{>0})$ by \cite{ChekanovGF,PanRuther}
    \[\begin{array}{ccc}
    ((J^1M \times \bR_{>0}) \times \bR, dw + s\alpha_{J^1M}) & \xrightarrow{\sim} & (J^1(M \times \bR_{>0}), \alpha_{J^1(M \times \bR_{>0})}),\\
    (x, \xi, t; s; w) & \mapsto & (x, s; s\xi, t; st + w),
    \end{array}\]
    where we use the standard contact form $\alpha_{J^1M} = dt - \xi dx$. The Legendrian lift of an exact Lagrangian cobordism from $\Lambda_-$ to $\Lambda_+$ gives rise to a conical Legendrian cobordism \cite{PanRuther}.

\begin{definition}
    Let $\Lambda_\pm \hookrightarrow J^1M$ be Legendrian embeddings. Then a Legendrian embedding $L \hookrightarrow J^1(M \times \bR_{>0})$ is called a conical Legendrian cobordism from $\Lambda_-$ to $\Lambda_+$ if for some $0 < s_- < s_+ < +\infty$, there exist $w_{0,-}, w_{0,+} \in \bR$ such that
    \begin{align*}
    {L} \cap J^1(M \times (0, s_-)) &= \{(x, s, s\xi, t, st + w_{0,-}) | (x, \xi, t) \in \Lambda_-, s \in (0, s_-)\}, \\
    {L} \cap J^1(M \times (s, +\infty)) &= \{(x, s, s\xi, t, st + w_{0,+}) | (x, \xi, t) \in \Lambda_+, s \in (s_+, +\infty)\}.
    \end{align*}
\end{definition}

    Given a generating family $f_-: M \times \bR^k \rightarrow \bR$ linear at infinity of $\Lambda_-$, the classical trick of Chekanov \cite{ChekanovGF}*{Proposition 5.3} implies that there is a generating family linear at infinity
    $$f_-: M \times \bR_{>0} \times \bR^k \rightarrow \bR, \;\, f_-(x, s, u) = sf_-(x, u)$$
    that generates the conical Legendrian $\Lambda_- \times \bR_{>0} = \{(x, s, s\xi, t, st) | (x, \xi, t) \in \Lambda_-\}$. This will be the starting point of our construction.

    In order to construct a generating family on the double copy $T_{-\epsilon}(L) \cup T_\epsilon(L)$, we need to consider the following function linear at infinity with two critical points.

\begin{definition}
    For $h > 0$, define $D_h: \bR \rightarrow \bR$ as a smooth function such that $D_h(w) = h(w^3 - 3w)/2$ when $|w| \leq 1$, $D'_h(w) \geq 0$ when $1 \leq |w| \leq 2$, and $D_h(w) = w$ when $|w| \geq 2$.
\end{definition}

\begin{lemma}
    Let $f: M \times \bR^k \rightarrow \bR$ be a generating family linear at infinity and $f_{dbl}(x, u, w) = f(x, u) + D_h(w)$. Then there is a homotopy of functions $f_{t}: M \times \bR^{k+1} \rightarrow \bR$, $0 \leq t \leq 1$, such that $\Sigma_{f_{t}} = \Sigma_{f_{dbl}}$, $f_t = f_{dbl}$ on a neighbourhood of the critical locus, and $f_1$ is linear at infinity.
\end{lemma}
\begin{proof}
    Suppose $f(x, u) = \epsilon(x, u) + L(x, u)$ where $L$ is linear and $\pi_M: \mathrm{supp}(\epsilon) \rightarrow M$ is proper. Consider a cut-off function $\rho: M \times \bR \rightarrow [0, 1]$ such that $\rho|_{M \times [-2, 2]} \equiv 1$ and $\pi_M: \mathrm{supp}(\rho) \to M$ is proper. Define 
    $$f_t = (t + (1 - t)\rho(x, w))\epsilon(x, u) + L(x, u) + D_h(w).$$
    We know that $\Sigma_{f_t} \cap (M \times \bR^k \times [-2, 2]) = \Sigma_{f_{dbl}}$. It suffices to show that $\Sigma_{f_t} \cap (M \times \bR^k \times (-\infty, -2) \cup (2, +\infty)) = \varnothing$. Actually, when $w \in (-\infty, -2) \cup (2, +\infty)$, we have
    $$\partial_w f_t = (1 - t) \partial_w \rho(x, w) \epsilon(x, u) + D_h'(w) = (1 - t) \partial_w \rho(x, w) \epsilon(x, u) + 1.$$
    Choose $\rho: M \times \bR \to \bR$ such that $\sup_{w \in \bR}|\partial_w \rho(x, w)| \leq \sup_{u \in \bR^k}|\epsilon(x, u)|$. Then $\partial_w f \neq 0$ for $w \in (-\infty, -2) \cup (2, +\infty)$, which means that $\Sigma_{f_t} \cap (M \times \bR^k \times (-\infty, -2) \cup (2, +\infty)) = \varnothing$. Thus, $\Sigma_{f_{t}} = \Sigma_{f_{dbl}}$ and $f_t$ is the required homotopy.
\end{proof}

    Given the above lemma, we will not distinguish $f_{dbl}(x, u, w) = f(x, u) + D_h(w)$ and its linear at infinity deformation.

\begin{theorem}\label{thm:double}
    Let $L \hookrightarrow J^1(M \times \bR_{>0})$ be an embedded conical Legendrian cobordism between closed embedded Legendrians from $\Lambda_-$ to $\Lambda_+ \hookrightarrow J^1M$.

    Then there exists a generating family linear at infinity $F_\epsilon: M \times \bR_{>0} \times \bR^{k+1} \rightarrow \bR$ on $T_{-\epsilon}(L) \cup T_\epsilon(L)$ for $\epsilon > 0$ sufficiently small if and only if there exists a generating family linear at infinity $f_-: M \times \bR^{k} \rightarrow \bR$ which extends to a germ of generating family on ${L}$.

    Moreover, when the conditions hold, we can define $F_\epsilon: M \times \bR_{>0} \times \bR^{k+1} \rightarrow \bR$ such that
    $$(F_\epsilon|_{M \times (0, s_-) \times \bR^k})(x, s, u, v, w) = s{f}_-(x, u) + D_\epsilon(w),$$
\end{theorem}

\begin{figure}
  \centering
  \includegraphics[width=0.8\textwidth]{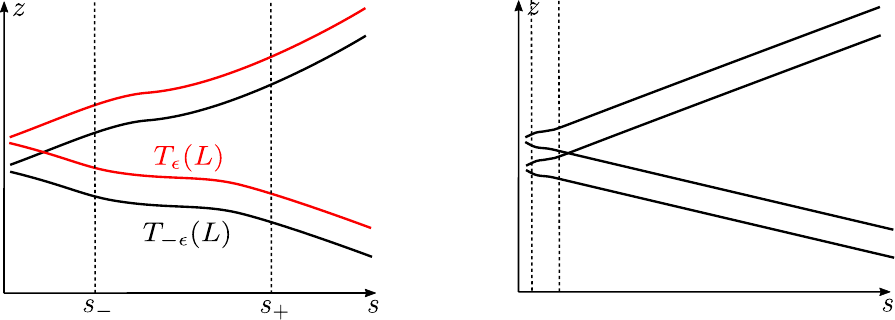}\\
  \caption{Left: the double copy Legendrian $T_{-\epsilon}(L) \cup T_\epsilon(L)$ in Theorem \ref{thm:double}. Right: the image of the double copy under the inverse Liouville flow $T_{-\epsilon}(\varphi'_{h/\epsilon}(L)) \cup T_\epsilon(\varphi'_{h/\epsilon}(L))$ in Theorem \ref{thm:largepushoff}. The regions between the dashed lines are the nonconical part of the Legendrian cobordism.}\label{fig:double}
\end{figure}

    Let us recall the theorem of Abouzaid-Courte-Guillermou-Kragh \cite{TwistGF}*{Corollary 3.12 \& Theorem 3.13}. They proved that given a germ of generating family on a closed Legendrian embedding $L \hookrightarrow J^1M$, there is a generating family linear at infinity on the doubling $T_{-\epsilon}(L) \cup T_\epsilon(L)$ when $\epsilon > 0$ is small. Their theorem does not hold for arbitrary noncompact Legendrian embeddings (even with conical conditions outside a compact subset).

    Consider the germ of generating families $F(x, s, u)$ where $F(x, s, u) = sf_-(x, u)$ at the negative end and $F(x, s, u) = sf_+(x, u)$ at the positive end. When $s > 0$ is sufficiently small, we will lose control on the derivative of $F(x, s, u) = sf_-(x, u)$, and may therefore introduce extra critical points to the generating family when considering the doubling construction. We avoid the issue by assuming the existence of a generating family linear at infinity at the negative end.

\begin{proof}[Proof of Theorem \ref{thm:double}]
    Consider the extension of the germ of generating family $f$ defined on an open neighbourhood $U$ of $L$ such that
    $$(f|_{M \times (0, s_-) \times \bR^k})(x, s, u, v) = sf_-(x, u).$$
    Since $L \hookrightarrow J^1(M \times \bR_{>0})$ is conical outside $J^1(M \times (s_-, s_+))$, we may choose the extension of the germ of functions on $M \times (s_+, +\infty) \times \bR^k$ such that
    $$(f|_{M \times (s_+, +\infty) \times \bR^k})(x, s, u, v) = sf_+(x, u, v)$$
    for some function $f_+: M \times \bR^k \rightarrow \bR$ defined on an open subset.

    Consider a open refinement $V \subset U$ of $L$ and a cut-off function $\alpha: M \times \bR_{>0} \times \bR^k \rightarrow [0, 1]$ such that $\alpha|_{M \times (0, s_-) \times \bR^k} \equiv 1$, $\alpha|_{V} \equiv 1$ and the projection
    $$\overline{\alpha^{-1}((0, 1])} \cap M \times [s_-, +\infty) \times \bR^k \rightarrow M \times [s_-, +\infty)$$
    is a proper map. Consider an arbitrary extension of $f$ from $V$ to $M \times \bR_{>0} \times \bR^k$ such that $f = 0$ in the complement of $U$. Define the function
    $$F_\epsilon(x, s, u, w) = f(x, s, u) + w + \alpha(x, s, u)(D_\epsilon(w) - w).$$
    We check that $F_\epsilon: M \times \bR_{>0} \times \bR^{k+1} \rightarrow \bR$ generates $T_{-\epsilon}(L) \cup T_\epsilon(L)$.

    When $s < s_-$, we know that
    $$F_\epsilon(x, s, u, w) = sf_-(x, u) + D_\epsilon(w).$$
    When $s_- \leq s \leq s_+$, considering the region $\overline{\alpha^{-1}((0, 1))} \cap M \times [s_-, s_+] \times \bR^k$, we may assume that $|\partial_{u}f| \neq 0$ and there exist uniform $c_1, c_2 > 0$ such that
    $$|\partial_{u}f| \geq c_1, \; |\partial_{u}\alpha| \leq c_2.$$
    Note that $D_\epsilon(w) - w \rightarrow 0$ when $\epsilon \rightarrow 0$. This shows that for $\epsilon > 0$ sufficiently small, $\partial_{u}f + (D_\epsilon(w) - w)\partial_{u}\alpha \neq 0$ when $\alpha(x, s, u) \neq 1$. Finally, when $s \geq s_+$,
    $$F_\epsilon(x, s, u, w) = sf_+(x, u) + w + \alpha(x, u, s)(D_\epsilon(w) - w).$$
    Since $\Lambda_+$ is closed, we may assume that the neighbourhood $U$ has positive radius with respect to the product metric on $M \times \bR_{>0} \times \bR^k$. Therefore, there exist uniform $c_1', c_2' > 0$ such that
    $$|\partial_{u}f_+| \geq c_1', \; |\partial_{u}\alpha| \leq c_2'.$$
    This shows that for $s > s_+$ and $\epsilon > 0$ sufficiently small, $|s\partial_u f_+| > |(D_{\epsilon}(w) - w) \partial_u \alpha|$, and hence $s\partial_{u}f_+ + (D_\epsilon(w) - w)\partial_{u}\alpha \neq 0$ when $\alpha(x, s, u) \neq 1$. This shows that the critical points of $F_\epsilon$ are located on the locus $\alpha^{-1}(1) \times \bR_w$ and hence generates  $T_{-\epsilon}(L) \cup T_\epsilon(L)$. Since
    $$\overline{f^{-1}_+((0, +\infty))} \cap M \times [s_-, +\infty) \times \bR^k \rightarrow M \times [s_-, +\infty)$$
    is proper, we can conclude that $F_\epsilon$ is linear at infinity.
\end{proof}
\begin{remark}
    One may notice that the uniform estimate $|\partial_{u}f_-| \geq c_1', \; |\partial_{u}\alpha| \leq c_2'$ still holds at the negative end. However, since $s \in (0, s_-)$ may be sufficiently small, we can no longer show that $|s\partial_u f_-| > |(D_{\epsilon}(w) - w) \partial_u \alpha|$.This is why we need a given generating family linear at infinity at the negative end to start with.
\end{remark}

    Since $\partial_u F = 0$ on the critical locus $\Sigma_F$, we can shrink the radius of the neighbourhood defining the germ $F$ to get the estimation on the derivatives. However, as the radius of the neighbourhood shrinks at the negative end, when trying to stack two copies of the germ $F$ using the doubling function $D_\epsilon$, there is going to be a gap between two copies of the neighbourhoods and it will be impossible to understand the behaviour in the gap. 

    On the contrary, when there is a Weinstein tubular neighbourhood of uniform positive radius, one can construct a generating family on the doubling of a non-uniform Reeb push-off and deform it to a generating family on a uniform Reeb push-off $T_{-\epsilon}(L) \cup T_\epsilon(L)$ inside the Weinstein tubular neighbourhood using contact isotopies. This obstruction has been explained by the author in the construction of microlocal sheaves \cite[Section 3]{LiCobordism2}.

\subsection{Homotopy lifting of generating families}
    We would like to separate the Legendrian front projection of $T_{-\epsilon}(L)$ and $T_\epsilon(L)$ in $M \times \bR_{>0} \times \bR$ using contact Hamiltonian isotopies in order to get a generating family on a single copy of $L$. This requires the homotopy lifting property of generating families.

    Homotopy lifting property for generating families of Lagrangian submanifolds is proved by Sikorav \cite{Sikorav}*{Proposition 1.2 \& 1.7} which also appears in \cite{GroEliashGF}*{Proposition 3.1.5}. In fact, the homotopy lifting is unique up to stabilization; see \cite{ViterboGen}*{Proposition 1.5}, \cite{TheretViterbo}*{Theorem 3.2} and \cite{JordanTraynor}*{Theorem 3.5}. Here we recall the case for Legendrian submanifolds.

\begin{theorem}[Chekanov \cite{ChekanovGF}*{Proposition 3.4, Lemma 5.6 \& Theorem 6.1}; Chaperon \cite{Chaperon}*{Theorem 4}; \cite{GroEliashGF}*{Proposition 4.1.1}; \cite{JordanTraynor}*{Theorem 1.2}]
    Let $\varphi: L \rightarrow J^1M$ be a Legendrian immersion and $f: M \times \bR^n \rightarrow \bR$ be a generating family for $L$ (linear at infinity). Let $\theta_\lambda$, $\lambda \in [0, 1]$, be a compactly supported contact isotopy on $J^1M$. Then there exist compactly supported functions $\eta_\lambda: M \times \bR^n \times \bR^{2m} \rightarrow \bR$ such that the generating family (linear at infinity)
    $$F_\lambda(x, u, p_1, q_1, \dots, p_m, q_m) = f(x, u) + p_1q_1 + \dots + p_mq_m + \eta_\lambda(x, u, p_1, q_1, \dots, p_m, q_m)$$
    generates $\theta_\lambda \circ \varphi: L \rightarrow J^1M$.
\end{theorem}
\begin{remark}
    In fact, whenever $\theta_\lambda$, $\lambda \in [0, 1]$, is a contact isotopy that has bounded $C^1$-norm with respect to some complete (adapted) metric, we can prove the above homotopy lifting property, but the function $\eta_\lambda: M \times \bR^n \times \bR^{2m} \rightarrow \bR$ will typically be non-compactly supported, which introduces extra difficulty to analyze the behaviour of the generating family at infinity. Therefore, we will only work with the compactly supported version.
\end{remark}

    In the similar construction for microlocal sheaves \cite[Section 3]{LiCobordism2}, following \cite{Gui} and \cite{TwistGF}, we separate the Legendrian fronts of $T_{-\epsilon}(L)$ and $T_\epsilon(L)$ by applying a contact isotopy that sends $T_{-\epsilon}(L) \cup T_\epsilon(L)$ by $T_{-h}(L) \cup T_h(L)$ for $r \gg 0$ using a cut-off of the Reeb flow. However, as $L$ is noncompact, the contact isotopy is not compact, and as explained in the remark, it will be difficult to control the behaviour of generating families at infinity.

    Our solution is to use the contact lift of the Liouville vector field on $J^1M \times \bR_{>0} \cong T^*(M \times \bR_{>0})$, which will push the non-conical part of the double copies of the Legendrian fronts $T_{-\epsilon}(L) \cup T_\epsilon(L)$ to the negative end where the fronts are already separated.

    Consider the local coordinates $(x, s; \xi, \sigma; z) \in J^1(M \times \bR_{>0})$. Consider the Liouville vector field $Z = s \partial_s$ on $J^1M \times \bR_{>0}$. This lifts to a contact vector field in $(J^1M \times \bR_{>0}) \times \bR$
    $$Z' = s \partial_s + w \partial_w,$$
    defined by the Hamiltonian $H'(x, \xi, t; s; w) = w.$
    Consider the contactomorphism $(J^1M \times \bR_{>0}) \times \bR \cong J^1(M \times \bR_{>0})$ and the coordinates $(x, s; y, t; z) \in J^1(M \times \bR_{>0})$. The contact vector field is then given by
    $$Z' = s \partial_s + y \partial_y + z \partial_z$$
    which is defined by the Hamiltonian $H'(x, s, y, t, z) = z - st$. We let $\varphi_\lambda: J^1M \times \bR_{>0} \to J^1M \times \bR_{>0}$ be the Liouville flow, and let $\varphi_\lambda' : J^1(M \times \bR_{>0}) \rightarrow J^1(M \times \bR_{>0})$ be the contact lifting of the Liouville flow defined by the Hamiltonian $H'$.

    Although the Liouville flow $\varphi_\lambda$ preserves the conical ends of exact Lagrangian cobordisms, in general it will change the primitive of the exact Lagrangian, so that the contact lifting $\varphi_\lambda'$ no longer preserves the conical ends of the Legendrian cobordisms. We implement this by introducing some other ambient contact isotopy.

    Suppose the conical Legendrian cobordism $L \hookrightarrow J^1(M \times \bR_{>0})$ has conical ends
    \begin{align*}
    L \cap J^1(M \times (0, s_-)) &= \{(x, s; s\xi, t; st + w_{0,-}) \mid (x, \xi, t) \in \Lambda_-, s \in (0, s_-) \}, \\
    L \cap J^1(M \times (s_+, \infty)) &= \{(x, s; s\xi, t; st + w_{0,+}) \mid (x, \xi, t) \in \Lambda_+, s \in (s_+, \infty) \}.
    \end{align*}
    Consider a smooth function $w_0: \bR_{>0} \to \bR$ such that $w_0(s) = w_{0,-}$ when $s < s_-$ and $w_0(s) = w_{0,+}$ when $s > s_+$. Let $T_{\lambda w_0}$ be the contact flow induced by the Hamiltonian $w_0(s)$ for time $\lambda \in \bR$, whose contact vector field is
    $$X_{w_0} = w_0'(s) \partial_t + w_0(s) \partial_z.$$
    Note that in fact $T_{\lambda w_0}$ commutes with the Reeb flow $T_\lambda$.

\begin{lemma}\label{lem:isotopy}
   Let $\varphi_\lambda$ be the Liouville flow and $\varphi_\lambda'$ be the contact lifting of the Liouville flow. Let $L \hookrightarrow J^1(M \times \bR_{>0})$ be a conical Legendrian cobordism with no Reeb chords. Then there exists a compactly supported Hamiltonian isotopy between $T_{-\epsilon}(L) \cup T_\epsilon(L)$ and $T_{-\epsilon} \circ T_{(1 - e^\lambda)w_0}(\varphi_{-\lambda}'(L)) \cup T_\epsilon \circ T_{(1 - e^\lambda)w_0}(\varphi_{-\lambda}'(L))$ for any given $\lambda \in \bR$.
\end{lemma}
\begin{proof}
    First, we show that $T_{-\epsilon} \circ T_{(1 - e^\lambda)w_0}(\varphi_{-\lambda}'(L)) \cup T_\epsilon \circ T_{(1-e^\lambda)w_0}(\varphi_{-\lambda}'(L))$ defines a Legendrian isotopy. Since this is obviously a family of Legendrian immersions, it suffices to show that this is always a family of Legendrian embedding. Consider the projection $\pi: J^1(M \times \bR_{>0}) \to T^*(M \times \bR_{>0})$. Then since $\varphi_\lambda'$ is the contact lifting of $\varphi_\lambda$,
    $$\pi \circ \varphi_{-\lambda}'(L) = \varphi_{-\lambda} \circ \pi(L).$$ 
    In particular, since $\pi(L)$ is an embedded Lagrangian, there are no Reeb chords on the Legendrian lift $\varphi_{-\lambda}'(L)$. This implies that $T_{-\epsilon}(\varphi_{-\lambda}'(L)) \cap T_\epsilon(\varphi_{-\lambda}'(L)) = \varnothing$, and hence $T_{-\epsilon} \circ T_{(1 - e^\lambda)w_0}(\varphi_{-\lambda}'(L)) \cap T_\epsilon \circ T_{(1 - e^\lambda)w_0}(\varphi_{-\lambda}'(L)) = \varnothing$. We thus get a Legendrian isotopy.

    Then, we prove show the Legendrian isotopy extends to a compactly supported Hamiltonian isotopy. We can compute that
     \begin{equation*}
    \resizebox{\textwidth}{!}{$T_{(1-e^\lambda)w_0} (\varphi_{-\lambda}'(x, s; y, t; z)) = \Big(x, e^{-\lambda} s; e^{-\lambda} y, t + (1 - e^{-\lambda}) w'_0(e^{-\lambda} s); w_0(e^{-\lambda} s) + e^{-\lambda} (z - w_0(e^{-\lambda} s)) \Big).$}
    \end{equation*}
    Since $w_0(s) = w_{0,-}$ for $s < s_-$ and $w_0(s) = w_{0,+}$ for $s > s-+$, we know that  for $s_-' < s_-$ sufficiently small and $s_+' > s_+$ sufficiently large,
    $$T_{(1-e^\lambda)w_0}(\varphi_{-\lambda}'(L)) \cap J^1(M \times (0, s_-') \cup (s_+', \infty)) = L \cap J^1(M \times (0, s_-') \cup (s_+', \infty)).$$
    Thus the isotopy $T_{-\epsilon} \circ T_{(1 - e^{\lambda})w_0}(\varphi_{-\lambda}'(L)) \cup T_\epsilon \circ T_{(1 - e^{\lambda})w_0}(\varphi_{-\lambda}'(L))$ is fixed outside a compact subset. We can cut off the vector field that generates the isotopy such that it is compactly supported. Since the Legendrian isotopy is generated by a compactly supported vector field, one can extend the compactly supported vector field from the Legendrian to its Weinstein neighbourhood as in the proof of the isotopy extension theorem \cite[Theorem 2.6.2]{Geigesbook}, which gives the compactly supported Hamiltonian isotopy.
\end{proof}

    After all the non-conical part of the Legendrian cobordism is pushed to the negative end by a compatly supported Hamiltonian, it suffices to rescale a neighbourhood of the negative end $J^1(M \times (0, s_-))$ to $J^1(M \times (0, \lambda s_-))$.

\begin{lemma}\label{lem:rescale}
    Let $L \hookrightarrow J^1(M \times \bR_{>0})$ be a Legendrian submanifold with generating family $F(x, s, u)$. Consider the contact isotopy $\phi_\lambda: J^1(M \times \bR_{>0}) \to J^1(M \times \bR_{>0})$ given by
    $$\phi_\lambda: (x, s; y, t; z) \mapsto \Big(x, e^\lambda s; e^\lambda y, t + (1 - e^\lambda) w'_0(e^\lambda s); w_0(e^\lambda s) + e^\lambda (z - w_0(e^\lambda s)) \Big).$$
    Then $\phi_\lambda(L)$ is a Legendrian with generating family $w_0(s) + e^\lambda(F(x, e^{-\lambda}s, u) - w_0(s))$.
\end{lemma}

\begin{theorem}\label{thm:largepushoff}
    Let $L \hookrightarrow J^1(M \times \bR_{>0})$ be a conical Legendrian cobordism with no Reeb chords between closed embedded Legendrians from $\Lambda_-$ to $\Lambda_+ \hookrightarrow J^1M$.

    Then there exists a generating family linear at infinity $F_h: M \times \bR_{>0} \times \bR^{K+1} \rightarrow \bR$ on $T_{-h}(L) \cup T_h(L)$ for any $h > 0$ if and only if there is a generating family linear at infinity $f_-: M \times \bR^{k} \rightarrow \bR$ on $\Lambda_-$ which extends to ${L}$.

    Moreover, when the conditions hold, we can define $F_h: M \times \bR_{>0} \times \bR^{K+1} \rightarrow \bR$ such that
    $$(F_h|_{M \times (0, s_-) \times \bR^K})(x, s, u, v, w) = s\overline{f}_-(x, u, v) + D_h(w),$$
    where $\overline{f}_-(x, u, v)$ is a stabilization of the function $f_-(x, u)$.
\end{theorem}
\begin{proof}
    Since by Lemma \ref{lem:isotopy} there is a compactly supported Hamiltonian isotopy between $T_{-\epsilon}(L) \cup T_\epsilon(L)$ and $T_{-\epsilon} \circ T_{(1 - e^\lambda)w_0}(\varphi_{-\lambda}'(L)) \cup T_\epsilon \circ T_{(1 - e^\lambda)w_0}(\varphi_{-\lambda}'(L))$, the homotopy lifting property of generating families implies that there exists a generating family linear at infinity on $T_{-\epsilon} \circ T_{(1 - e^\lambda)w_0}(\varphi_{-\lambda}'(L)) \cup T_\epsilon \circ T_{(1 - e^\lambda)w_0}(\varphi_{-\lambda}'(L))$. Then consider the contactomorphism $\phi_\lambda: J^1(M \times \bR_{>0}) \xrightarrow{\sim} J^1(M \times \bR_{>0})$ in Lemma \ref{lem:rescale}. We can show that
    $$\phi_{\,\ln(h/\epsilon)}: T_{-\epsilon} \circ T_{(1-{h/\epsilon})w_0}(\varphi_{-\ln(h/\epsilon)}'(L)) \cup T_\epsilon \circ T_{(1-{h/\epsilon})w_0}(\varphi_{-\ln(h/\epsilon)}'(L)) \xrightarrow{\sim} T_{-h}(L) \cup T_h(L).$$
    Thus, by Theorem \ref{thm:double} and Lemma \ref{lem:rescale}, since $w_0(s) = w_{0,-}$ when $s < s_-$, we can get a generating family $F_h$ linear at infinity on $T_{-h}(L) \cup T_h(L)$ such that
    $$(F_h|_{M \times (0, s_-) \times \bR^{K+1}})(x, s, u, v, w) = sf_-(x, u) + Q(v) + D_h(w).$$
    Then we deform the generating family on $M \times (0, s_-) \times \bR^{K+1}$ without changing the critical locus through the homotopy $F_{h,t}|_{M \times (0, s_-) \times \bR^{K+1}} = sf_-(x, u) + \rho(s)Q(v) + D_h(w)$ for some smooth function $\rho: \bR_{>0} \rightarrow \bR_{>0}$ such that $\rho(s) = s$ when $s$ is small and $\rho(s) = 1$ when $s$ is close to $s_-$. Hence we may assume that
    $$(F_h|_{M \times (0, s_-') \times \bR^{K+1}})(x, s, u, v, w) = s(f_-(x, u) + Q(v)) + D_h(w)$$
    for some $s_-' < s_-$. This finishes the proof of the theorem.
\end{proof}
\begin{remark}
    We would like to thank the anonymous referee for pointing out the mistake in the previous version and suggesting revisions on the argument of the above theorem.
\end{remark}

\subsection{Cutting off generating families}\label{sec:cutoff}
    In this section, we will construct the generating family linear at infinity on $L$ from a generating family linear at infinity on $T_{-h}(L) \cup T_h(L)$ by cutting off the domain. The main technique is a combination of Kragh \cite[Section 3]{KraghR2n} and Abouzaid-Courte-Guillermou-Kragh \cite[Section 3.4]{TwistGF}.

    First, we explain how to extend generating families equivariantly on a conical Legendrian.

\begin{lemma}\label{lem:equiv-extend}
    Let $L \hookrightarrow J^1(M \times \bR_{>0})$ be a Legendrian embedding conical outside $J^1(M \times [s_-, s_+])$ whose projection onto $M \times (0, s_-) \cup (s_+, +\infty)$ is proper. Suppose there is a generating family (linear at infinity) $F_0: M \times [s_-', s_+'] \times \bR^k \rightarrow \bR$ for $L \cap J^1(M \times [s_-', s_+'])$ for some $s_-' < s_- < s_+ < s_+'$. Then there is a generating family (linear at infinity) $F: M \times \bR_{>0} \times \bR^k \rightarrow \bR$ for $L$ such that
    \begin{enumerate}
      \item $(F|_{M \times (0, s_-')})(x, s, u) = sf_-(x, u)$;
      \item $(F|_{M \times (s_+', +\infty)})(x, s, u) = sf_+(x, u)$;
      \item $(F|_{M \times [s_-,s_+]})(x, s, u) = F_0(x, s, u)$.
    \end{enumerate}
\end{lemma}
\begin{proof}
    We only consider the positive end. The proof for the negative end is similar. Since the projection of the conical ends of $L$ onto $(0, s_-) \cup (s_+, +\infty)$ is proper, we may assume that $M$ is compact. Write $f_+(x, u) = s_+^{-1}F_0(x, s_+, u)$. Then, in a neighbourhood $U$ of the critical locus in $M \times (s_+, +\infty) \times \bR^{K+1}$, by the parametric Morse lemma \cite[Appendix I \& Chapter IV Proposition 5.4]{GuiStern}, we can apply a fiberwise diffeomorphism within the small neighbourhood $U$ such that 
    $$F_0(x, s, u) = s f_+(x, u), \; (x, s, u) \in U.$$
    Indeed, since $F_0(x, s, u) - sf_+(x, u)$ vanishes up to the second order and the Hessians have the same rank and signature, we can connect them by a smooth family of functions and then define the fiberwise diffeomorphism by Moser's argument. Let $\rho: \bR_{>0} \rightarrow \bR$ be a cut-off function such that $\rho|_{(0,s_+]} \equiv 1$ and $\rho|_{[s_+', +\infty)} \equiv 0$. Define
    $$F(x, s, u) = \rho(s)F_0(x, s, u) + (1 - \rho(s))s f_+(x, u).$$
    When $s_+'$ is sufficiently close to $s_+$, we claim that the critical locus of $F$ on $M \times [s_+, s_+']$ necessarily agrees with $F_0$. In fact, since $f_+(x, u) = s_+^{-1}F_0(x, s_+, u)$, the critical points of $F$ are the locus determined by the equation
    $$\rho(s) \partial_u F_0(x, s, u) + (1 - \rho(s))s s_+^{-1} \partial_u F_0(x, s_+, u) = 0.$$
    Since $F_0|_{U} = s f_+|_{U}$, there are no extra critical points inside $U$. For any extra critical point outside $U$, the unit gradient vectors of $\partial_u F_0(x, s, u)$ and $\partial_u F_0(x, s_+, u)$ are necessarily in opposite directions. Note that $M$ is compact and $F_0$ is linear at infinity,  we may assume that the gradient vector field $\partial_u F_0(x, s_+, u)$ is bounded from below outside $U$. Then by choosing $s_+'$ sufficiently close to $s_+$, we know that for any $(x, s, u) \notin U$,
    $$\partial_u F_0(x, s, u)/|\partial_u F_0(x, s, u)| \neq -\partial_u F_0(x, s_+, u)/|\partial_u F_0(x, s_+, u)|,$$
    which implies that there are no extra critical points outside $U$. Therefore, $F$ defines a generating family linear at infinity on $M \times \bR_{>0}$ that is equivariant on both ends.
\end{proof}

    Theorem \ref{thm:largepushoff} ensures the existence of a generating family linear at infinity on $T_{-h}(L) \cup T_h(L)$ for any $h > 0$. When $h \gg 0$ the front projection of $T_{-h}(L) \cap J^1(M \times (0, s_+))$ and $T_h(L) \cap J^1(M \times (0, s_+))$ are separated by a hypersurface $z = 0$. Then by considering the sublevel set $E = F^{-1}((-\infty, 0))$, one will get a generating family on $L$ defined over $E$. Note that a similar approach has appeared in \cite[Section 3]{KraghR2n} and \cite[Section 3.4]{TwistGF}.

    We will assume that the generating family is strongly linear at infinity. Therefore, on each fiber $\bR^{K+1}$ the generating function which defines a parametrized cobordism from $\bR^K$ to $\bR^K$ can be viewed as a parametrized cobordism of compact disks from $\bD^K$ to $\bD^K$. 

    Deducing the following lemma from Theorem \ref{thm:largepushoff} is basically a standard exercise in differential topology. We will restrict the function from $M \times (0, s_+] \times \bR^{K+1} \to M \times (0, s_+]$ to a subbundle $E|_{M \times (0, s_+]} \to M \times (0, s_+]$, whose fiber is an open subset in $\bR^{K+1}$ with smooth boundary. We call such a generating family linear at infinity if it is the restriction of a linear function on $\bR^{K+1}$ outside a compact subset.

\begin{lemma}\label{lem:gf-fiberbundle}
    Let $L \rightarrow J^1(M \times \bR_{>0})$ be a conical Legendrian cobordism between closed Legendrians from $\Lambda_-$ to $\Lambda_+ \hookrightarrow J^1(M)$. Given a generating family linear at infinity $f_-: M \times \bR^{k} \rightarrow \bR$ on $\Lambda_-$, suppose the germ of generating family on $\Lambda_-$ extends to ${L}$. Then there exists a generating family $F: E|_{M \times (0, s_+]} \rightarrow \bR$ for $L \cap J^1(M \times (0, s_+])$ linear at infinity with outward pointing gradient vector field on the boundary, where $\pi: E|_{M \times (0, s_+]} \rightarrow M \times (0, s_+]$ is a smooth fiber bundle.
\end{lemma}
\begin{proof}
    Let $h > 0$ be sufficiently large. By Theorem \ref{thm:largepushoff}, we obtain a generating family linear at infinity $F: M \times \bR_{>0} \times \bR^{K+1} \rightarrow \bR$ on $T_{-h}(L) \cup T_h(L)$ such that
    $$(F|_{M \times (0, s_-] \times \bR^{K+1}})(x, s, u, v, w) = s\overline{f}_-(x, u, v) + D_h(w).$$
    When $h > 0$ is sufficiently large, we may assume that the front projection of $T_{-h}(L)$ and $T_h(L)$ inside $M \times \bR \times (0, s_+']$ are separated by a hyperplane $z = 0$. Let $E = F^{-1}((-\infty, 0))$. Then we obtain a generating family $F: E|_{M \times (0,s_+]} \rightarrow \bR$ for $L \cap J^1(M \times (0, s_+])$ linear at infinity with outward pointing gradient vector field on $\partial E|_{M \times (0, s_+]} = F^{-1}(0)|_{M \times (0, s_+]}$.

    We prove that $\pi: E|_{(0, s_+]} \rightarrow M \times (0, s_+]$ is a smooth manifold bundle. For any $(x_0, s_0) \in M \times (0, s_+]$, we show that there exists a neighbourhood $D_{(x_0, s_0)} \subset M \times (0, s_+]$ such that $E|_{D_{(x_0, s_0)}}$ is isomorphic to $E_{(x_0, s_0)} \times D_{(x_0, s_0)}$ as smooth manifold bundles. Write $F_{(x, s)}(u, v, w) = F(x, s, u, v, w)$. Since $z = 0$ is a regular value of $F_{(x_0, s_0)}$ which is linear at infinity, there is a tubular neighbourhood of $F_{(x_0, s_0)}^{-1}(0)$ that is diffeomorphic to
    $$F_{(x_0, s_0)}^{-1}(0) \times (-\epsilon, \epsilon) \cong F_{(x_0, s_0)}^{-1}((-\epsilon, \epsilon)) \subset \bR^{K+1}.$$
    When $0$ is a regular value, $F_{(x, s)}^{-1}(0)$ defines a smooth family of hypersurfaces in $\bR^{K+1}$. Hence when $\delta > 0$ is sufficiently small, for $(x, s) \in D_{(x_0, s_0)}(\delta)$, $F_{(x, s)}^{-1}(0)$ is a graphical hypersurface inside the tubular neighbourhood $F_{(x_0, s_0)}^{-1}(0) \times (-\epsilon, \epsilon)$. Then, by reparametrizing the sublevel set in the neighbourhood $F_{(x_0, s_0)}^{-1}((-\epsilon, \epsilon)) \cap F_{(x,s)}^{-1}((-\infty, 0))$, we can get a local trivialization
    $$(F|_{D_{(x_0, s_0)} \times \bR^{K+1}})^{-1}(-\infty, 0) \cong F_{(x_0, s_0)}^{-1}((-\infty, 0)) \times D_{(x_0, s_0)},$$
    and hence locally $E|_{D_{(x_0, s_0)}} \cong E_{(x_0, s_0)} \times D_{(x_0, s_0)}$ as smooth fiber bundles.
\end{proof}

    Moreover, it is standard in Morse theory to deduce that the fiber of the smooth manifold bundle is $\bR^{K+1}$. The natural question then, is whether the fiber bundle $\pi: E \rightarrow M \times \bR_{>0}$ is the trivial bundle $\pi_{M \times \bR_{>0}}: M \times \bR_{>0} \times \bR^{K+1} \rightarrow M \times \bR_{>0}$. 

    In the case we are working on right now, we have a trivialization near the negative end $\pi: E|_{M \times (0, s_-)} \rightarrow M \times (0, s_-)$ for free from the given generating function we start with, which will extend to a trivialization on the whole fiber bundle. Note that a similar approach is used by Kragh \cite[Section 3]{KraghR2n} to construct a generating family quadratic at infinity for exact Lagrangians in $T^*\bD^n$.

\begin{theorem}
    Let $L \hookrightarrow J^1(M \times \bR_{>0})$ be an embedded conical Legendrian cobordism with no Reeb chords between closed Legendrians from $\Lambda_-$ to $\Lambda_+ \hookrightarrow J^1(M)$. Given a generating family linear at infinity $f_-: M \times \bR^{k} \rightarrow \bR$ on $\Lambda_-$, suppose the germ of generating family for $\Lambda_-$ extends to ${L}$. Then there exists a generating family linear at infinity $F: M \times \bR_{>0} \times \bR^{K+1} \rightarrow \bR$ on $L$ such that
    $$(F|_{M \times (0, s_-) \times \bR^{K+1}})(x, s, u, v, w) = s\overline{f}_-(x, u, v, w),$$
    where $\overline{f}_-(x, u, v, w)$ is a stabilization of the function $f_-(x, u)$.
\end{theorem}
\begin{proof}
    Let $s_-' < s_- < s_+ < s_+'$. Using Lemma \ref{lem:gf-fiberbundle}, we can get a a generating family $F: E|_{M \times [s_-', s_+']} \rightarrow \bR$ for $L \cap J^1(M \times [s_-', s_+'])$. We show that $\pi: E|_{[s_-', s_+']} \rightarrow M \times [s_-', s_+']$ is a trivial manifold bundle whose fiber is diffeomorphic to $\bR^{K+1}$. Since $M \times [s_-', s_-] \hookrightarrow M \times [s_-', s_+']$ is a homotopy equivalence, it suffices to show that the fiber bundle $\pi: E|_{M \times [s_-', s_-]} \rightarrow M \times [s_-', s_-]$ is a trivial fiber bundle with fiber $\bR^{K+1}$.

    We consider $z \gg 0$ such that $s \overline{f}_-|_{M \times [s_-', s_-]}$ is linear when $s \overline{f}_-(x, u, v) \geq z$. Then define $H_{M \times [s_-', s_-]} = (s \overline{f}_-|_{M \times [s_-', s_-]})^{-1}((-\infty, z))$ to be a fiber bundle containing all the descending manifolds of $s \overline{f}_{-,x}(u, v) = s \overline{f}_-(x, u, v)$ for $(x, s) \in M \times [s_-', s_-]$, whose fibers are half planes $\bR^K_{<0}$. Since the gradient vector field of $s \overline{f}_{-,x}(u, v)$ on the complement of $H_{M \times [s_-', s_-]}$ is linear with no critical points, $H_{M \times [s_-', s_-]} \to M \times [s_-', s_-]$ is a trivial fiber bundle. 

    Since the critical points of $D_h(w)$ are $w = \pm 1$, we know that $H_{M \times [s_-', s_-]} \times 1 \subseteq \bR^{K+1}$ contains all the descending manifolds of $F_{(x,s)}(u, v, w) = s f_{-}(x, u, v) + D_h(w) + h$ for the critical points on $w = 1$, which generates the Legendrian $L$. Consider $z' \gg 0$ such that $F|_{M \times [s_-', s_-]}$ is linear when $F(x, s, u, v, w) \leq -z'$. Define $W_{M \times [s_-', s_-]} \subseteq M \times [s_-', s_-] \times \bR^{K+1}$ to be a neighbourhood of $(H_{M \times [s_-', s_-]} \times 1) \cup (F|_{(M \times [s_-', s_-]})^{-1}((-\infty, -z'))$ whose boundary is transverse to the gradient vector field. Since $H_{M \times [s_-', s_-]} \to M \times [s_-', s_-]$ is a trivial fiber bundle whose fibers are $k$-half planes, it follows that $W_{M \times [s_-', s_-]} \to M \times [s_-', s_-]$ is also a trivial fiber bundle whose fibers are $(k+1)$-half planes. Moreover, by the parametric Morse lemma \cite[Appendix I \& Chapter IV Proposition 5.4]{GuiStern}, after applying a fiberwise diffeomorphism in $W_{M \times [s_-', s_-]}$ without changing the critical locus, we may assume that
    $$(F|_{W_{M \times [s_-', s_-]}})(x, s, u, v, w) = s\overline{f}_-(x, u, v) + s (w-1)^2$$
    is a generating family for $L$. Indeed, since the difference of $s\overline{f}_-(x, u, v) + h(w^3 - 2w)/2 + h$ and $s\overline{f}_-(x, u, v) + s (w-1)^2$ vanish up to the second order and the Hessians have the same rank and signature, we can consider a deformation between the functions in the open subset $W_{M \times [s_-', s_-]}$ and construct the fiberwise diffeomorphism by Moser's argument. 

    On the other hand, each fiber of $E|_{M \times [s_-', s_-]} \backslash W_{M \times [s_-', s_-]}$  is a trivial cobordism from the boundary of the plane $\bR^K$ to $\bR^K$ since the gradient vector field of the family of functions is complete with no critical points. Hence we conclude that $E|_{M \times [s_-', s_-]} \rightarrow M \times [s_-', s_-]$ is a trivial fiber bundle with fiber being the $(k+1)$-half plane. Since the gradient vector field of $F_{(x,s)}$ is outward pointing on the boundary $\partial E|_{M \times [s_-', s_+']}$ whose fibers are diffeomorphic to $\bR^K$, we can extend the vector field so that it is linear at infinity. Moreover, since the gradient vector field on $E|_{M \times [s_-', s_-]} \backslash W_{M \times [s_-', s_-]}$ is complete with no critical points, after applying a fiberwise diffeomorphism without changing the critical locus, we may assume that
    $$(F|_{M \times [s_-', s_-] \times \bR^{K+1}})(x, s, u, v, w) = s\overline{f}_-(x, u, v) + s w^2$$
    is a generating family on $L$. Thus we trivialize $\pi: E|_{M \times [s_-', s_+']} \rightarrow M \times [s_-', s_+']$ by extending the trivialization from $M \times [s_-', s_-]$ to $M \times [s_-', s_+']$. Finally, we extend the generating family equivariantly from $M \times [s_-', s_+']$ to $M \times \bR_{>0}$ using Lemma \ref{lem:equiv-extend}.
\end{proof}

\begin{figure}
  \centering
  \includegraphics[width=0.7\textwidth]{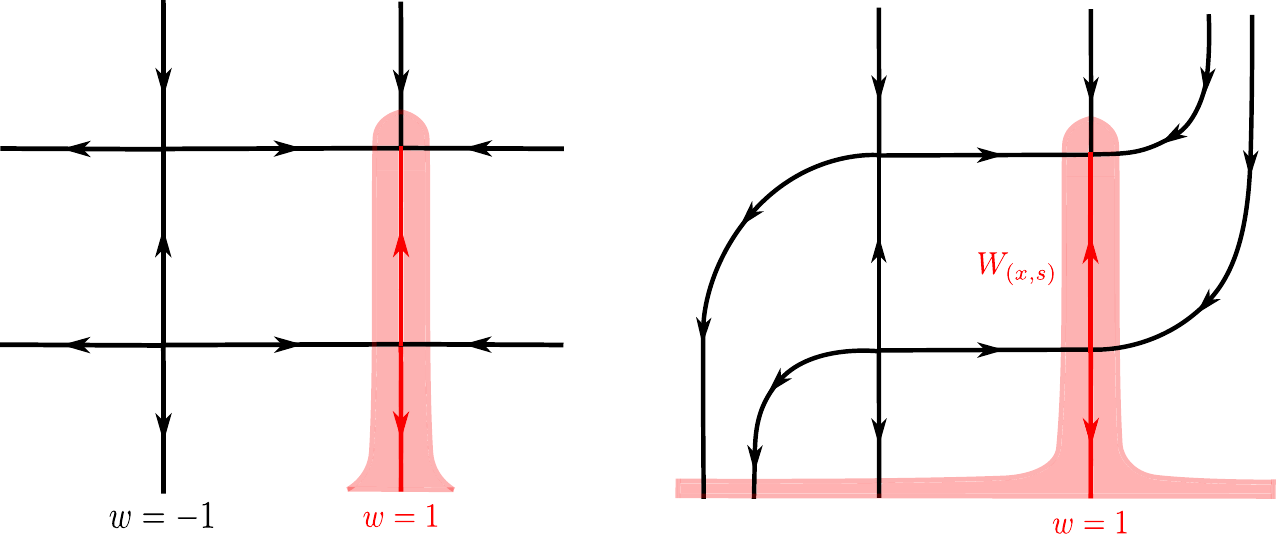}\\
  \caption{The negative gradient vector field of the generating family $F(x, s, u, v, w)$. The red region on $w = 1$ is the fiber of the bundle $H_{M \times [s_-', s_-]}$ and the pink region around it is the fiber of the bundle $W_{M \times [s_-', s_-]}$. The function on the left hand side is $s f_-(x, u, v) + D_h(w)$ while the function on the right hand side is the linear at infinity modification.}\label{fig:cutoff}
\end{figure}

\appendix
\section{The Case of Legendrian Links}\label{sec:knot}
    We write down in detail the necessary sufficient condition of extending generating families in the case of 1-dimensional Legendrian links.

\begin{definition}
    Let $\Lambda \subset J^1M$ be a Legendrian with vanishing Maslov class and generic front projection, whose singular locus under the front projection is $\Lambda_\text{sing}$. A Maslov potential is a map $d: \Lambda \backslash \Lambda_\text{sing} \rightarrow \bZ$ such that for any generic curve $\gamma: [0, 1] \rightarrow \Lambda$, $d(\gamma(1)) - d(\gamma(0))$ equals the number of upward going cusps minus the number of downward going cusps.

    Let $f: M \times \bR^N \rightarrow \bR$ be a generating family for the Legendrian $\Lambda \subset J^1M$ with generic front projection, whose singular locus under the front projection is denoted by $\Lambda_\text{sing}$. The Maslov potential associated to the generating family $f$ is the map $d(f): \Lambda \backslash \Lambda_\text{sing} \rightarrow \bZ$ that sends a point to the Morse index of the generating family at the point.
\end{definition}

    For a Legendrian $\Lambda \subset J^1M$, by fixing a reference Maslov potential $d_0: \Lambda \backslash \Lambda_\text{sing} \rightarrow \bZ$, any Maslov potential corresponds uniquely to a continuous map $\Lambda \rightarrow \bZ$. If we moreover subtract the Maslov potential by the potential at the base point $p_0 \in \Lambda$, then the potential associated to a generating family is invariant under stabilizations, and the potential corresponds uniquely to a pointed map $\Lambda \rightarrow \bZ$. This agrees with the discrete part of the classifying map $\Delta(f): \Lambda \rightarrow \bZ \times BO$.

\begin{proposition}
    Let $L \hookrightarrow T^*(\bR^1 \times \bR_{>0})$ or $T^*(S^1 \times \bR_{>0})$ be an exact Lagrangian cobordism between Legendrian links in $J^1\bR^1$ or $J^1S^1$ from $\Lambda_-$ to $\Lambda_+$ and $\Lambda_+ \neq \varnothing$. Let $f_-$ be a generating family of $\Lambda_-$ linear at infinity. Then $f_-$ extends to a generating family of $L$ linear at infinity up to stabilizations if and only if the Maslov class $\mu(L) \in H^1(L; \bZ)$ is trivial and Maslov potential associated to $f_-$ extends to $L$.
\end{proposition}
\begin{proof}
    First, we show that the Lagrangian Gauss map $L \rightarrow U/O$ is trivial if and only if the Maslov class vanishes, i.e.~$\mu(L) = 0$. By Gromov's theorem, there are no closed exact Lagrangians in $T^*(\bR^1 \times \bR_{>0})$ or $T^*(S^1 \times \bR_{>0})$. Hence $L$ is a surface with boundary and has the homotopy type of a 1-dimensional CW complex. The Lagrangian Gauss map is nontrivial if and only if it is trivial along any 1-cycle $S^1 \rightarrow L \rightarrow U/O$, in other words, the Maslov class vanishes.

    Next, we show that $\Delta(f_-): \Lambda_- \rightarrow \bZ \times BO$ extends to $L$ if and only if the Maslov potential defined $d(f_-): \Lambda_- \rightarrow \bZ$ extends to $L$. When $\Lambda_+ \neq \varnothing$, $(L, \Lambda_-)$ has the homotopy type of a relative 1-dimensional CW complex, and the existence of an extension of $\Delta(f_-)$ from $\Lambda_-$ to $L$ is equivalent to the existence of an extension of $\Delta(f_-)$ along any arc $(D^1, S^0) \rightarrow (L, \Lambda_-)$. However, we know that $BO$ is path connected. Hence this is equivalent to the existence of an extension of the Maslov potential.
\end{proof}

\begin{proposition}
    Let $L \hookrightarrow T^*(\bR^1 \times \bR_{>0})$ or $T^*(S^1 \times \bR_{>0})$ be an exact Lagrangian cap between Legendrian links in $J^1\bR^1$ or $J^1S^1$ from $\Lambda_-$ to $\Lambda_+ = \varnothing$. Then $\Lambda_-$ does not admit a generating family linear at infinity.
\end{proposition}
\begin{proof}
    By Dimitroglou Rizell's result, we know that when $L$ is a Lagrangian cap, $\Lambda_-$ cannot have any augmentation over $\bZ/2\bZ$ \cite{RizellCap}. Then, for Legendrian links in $J^1\bR^1$ or $J^1S^1$, by Fuchs and Sabloff's result, $\Lambda_-$ cannot have any graded normal ruling \cite{FuchsRuling,SabRuling}, and thus by Fuchs-Rutherford's result cannot have any generating family linear at infinity \cite{FuchsRutherGF}.
\end{proof}

\bibliographystyle{amsplain}
\bibliography{generatingfamily}
\end{document}